\documentclass[11pt,oneside,english]{amsart}

\usepackage[utf8]{inputenc}
\usepackage{amssymb}
\usepackage[unicode=true,pdfusetitle,
 bookmarks=true,bookmarksnumbered=false,bookmarksopen=false,
 breaklinks=false,pdfborder={0 0 1},backref=false,colorlinks=false]
 {hyperref}

 \newtheorem{thm}{Theorem}[section]
 \newtheorem{lem}[thm]{Lemma}
 \newtheorem{prop}[thm]{Proposition}
 \numberwithin{equation}{section}

\begin{document}

\title[Quantum Ergodicity for a Point Scatterer on the 3D Torus]
{Quantum Ergodicity for a Point Scatterer on the Three-Dimensional
Torus}

\author{Nadav Yesha}
\begin{abstract}
Consider a point scatterer (the Laplacian perturbed by a delta-potential)
on the standard three-dimensional flat torus. Together with the eigenfunctions
of the Laplacian which vanish at the point, this operator has a set
of new, perturbed eigenfunctions. In a recent paper, the author was
able to show that all of the perturbed eigenfunctions are uniformly
distributed in configuration space. In this paper we prove that almost
all of these eigenfunctions are uniformly distributed in phase space,
i.e. we prove quantum ergodicity for the subspace of the perturbed
eigenfunctions. An analogue result for a point scatterer on the two-dimensional
torus was recently proved by Kurlberg and Ueberschär.
\end{abstract}

\address{Raymond and Beverly Sackler School of Mathematical Sciences, Tel
Aviv University, Tel Aviv 69978, Israel}

\email{nadavye1@post.tau.ac.il}

\maketitle

\section{Introduction}

Consider a point scatterer on the standard three-dimensional flat
torus $\mathbb{T}^{3}=\mathbb{R}^{3}/2\pi\mathbb{Z}^{3}$, which is
formally given by
\begin{equation}
-\Delta+\alpha\delta_{x_{0}}\label{eq:Scatterer}
\end{equation}
where $-\Delta$ is the associated Laplacian on $\mathbb{T}^{3}$,
$\delta_{x_{0}}$ is the Dirac potential at $x_{0}$ and $\alpha$
is a coupling parameter.

We want to study quantum ergodicity of this system, which is a key
question in the field of Quantum Chaos. A classical result regarding
quantum ergodicity is Schnirelman's theorem \cite{Schnirelman,Colin2,Zelditch},
which asserts that for classically ergodic systems, the quantum counterpart
is quantum ergodic, i.e. almost all eigenstates are uniformly distributed
in phase space. However our system is not classically ergodic but
an intermediate system; its classical dynamics is nearly integrable.

Rigorously, a point scatterer is obtained as a self-adjoint extension
of the Laplacian acting on functions vanishing near $x_{0}.$ Such
extensions are parametrized by $\phi\in\left(-\pi,\pi\right]$, where
$\phi=\pi$ corresponds to the standard Laplacian ($\alpha=0$ in
\eqref{eq:Scatterer}). For $\phi\neq\pi,$ the eigenfunctions of
the corresponding operator consist of eigenfunctions of the Laplacian
which vanish at $x_{0}$ (and are not affected by the scatterer, so
they are related to the unperturbed problem), and new, perturbed eigenfunctions.
Since the latter are the only eigenfunctions which feel the scatterer,
we have to consider only them in order to study the physics of the
perturbed problem.

In a recent paper \cite{Yesha}, the author was able to prove quantum
(unique) ergodicity regarding the perturbed eigenfunctions of a point
scatterer on the standard three-dimensional flat torus, but only for
observables which do not depend on the momentum. Our goal now is to
prove quantum ergodicity regarding the perturbed eigenfunctions in
full phase space.

We remark that a two-dimensional analogue of the theorem in \cite{Yesha}
(i.e. quantum ergodicity in configuration space) was proved for a
general two-dimensional flat torus by Rudnick and Ueberschär \cite{Rudnick};
in a recent work by Kurlberg and Ueberschär, they prove an analogue
of the theorem in the current paper for the standard two-dimensional
flat torus.

The unit (co)tangent bundle of $\mathbb{T}^{3}$ is the compact metric
space $S^{*}\mathbb{T}^{3}\simeq\mathbb{T}^{3}\times S^{2}$, on which
we have the Liouville probability measure $\mu$ which is the normalized
product of the Lebesgue measure $m$ on $\mathbb{T}^{3}$ and the
Lebesgue measure $\sigma$ on $S^{2}.$ Observables are smooth, zero-th
order $\xi$-homogeneous functions $a\left(x,\xi\right)$ on $S^{*}\mathbb{T}^{3}$.
In order to quantize them, we use the notion of pseudo-differential
operators on $\mathbb{T}^{3}$, which will be discussed in greater
detail in section §\ref{sec:PDO} below. As we will see, quantizing
$a\in C^{\infty}\left(S^{*}\mathbb{T}^{3}\right)$ leads to an operator
$\mbox{Op}\left(a\right)$, which is a bounded operator on $L^{2}\left(\mathbb{T}^{3}\right).$

We now state the main theorem of this paper. For every $\phi\in\left(-\pi,\pi\right)$,
let $\Lambda_{\phi}$ be the set of the perturbed eigenvalues of the
point scatterer, with the corresponding $L^{2}$-normalized eigenfunctions
$g_{\lambda}$ $\left(\lambda\in\Lambda_{\phi}\right)$. We prove:
\begin{thm}
Fix $\phi\in\left(-\pi,\pi\right)$. There is a subset $\Lambda_{\phi,\infty}\subseteq\Lambda_{\phi}$
of density one so that for all \textup{$a\in C^{\infty}\left(S^{*}\mathbb{T}^{3}\right)$,
\[
\left\langle \mbox{Op}\left(a\right)g_{\lambda},g_{\lambda}\right\rangle \to\int_{S^{*}\mathbb{T}^{3}}a\mbox{d}\mu
\]
}as $\lambda\to\infty$ along $\Lambda_{\phi,\infty}.$
\end{thm}
We actually prove a more general statement: let $\mathcal{N}_{3}$
be the set of integers which are sums of three squares (these are
the eigenvalues of the Laplacian), and let $\Lambda$ be any increasing
sequence whose elements interlace with the elements of $\mathcal{N}_{3}$.
For any $\lambda\in\Lambda$, define $g_{\lambda}$ to be the $L^{2}$-normalized
Green's function: $g_{\lambda}=\frac{G_{\lambda}}{\left\Vert G_{\lambda}\right\Vert _{2}}$,
where $G_{\lambda}=\left(\Delta+\lambda\right)^{-1}\delta_{x_{0}}$.
\begin{thm}
\label{thm:MainTheroem}There is a subset $\Lambda_{\infty}\subseteq\Lambda$
of density one so that for all \textup{$a\in C^{\infty}\left(S^{*}\mathbb{T}^{3}\right)$,
\[
\left\langle \mbox{Op}\left(a\right)g_{\lambda},g_{\lambda}\right\rangle \to\int_{S^{*}\mathbb{T}^{3}}a\mbox{d}\mu
\]
}as $\lambda\to\infty$ along $\Lambda_{\infty}.$
\end{thm}
This is of interest since in the physics literature one considers
$\phi$ which is not fixed, but varies as $\lambda\to\infty$ (see
\cite{Shigehara,Shigehara2}); since the only condition for $\Lambda$
in theorem \ref{thm:MainTheroem} is that its elements interlace with the
elements of $\mathcal{N}_{3}$, the theorem will still hold in such
cases.

Approximating an observable with a linear combination of the functions
$Y_{l,m}\left(\xi\right)e^{ix\cdot\zeta}$, where $Y_{l,m}$ are spherical
harmonics, the question at hand reduces to an arithmetic one. The
main case is $\zeta=0$ (the other cases will follow from the proof
of proposition 3.9 in \cite{Yesha}), for which a new arithmetic ingredient
is used -- an estimate due to Duke, by which he proved in \cite{Duke,Duke2}
the equidistribution of integer lattice points on a sphere of radius
$\sqrt{n}$, $n\not\equiv0,4,7\,\left(8\right),\, n\to\infty$ (conjectured
by Linnik, and proved independently by Golubeva and Fomenko \cite{Golubeva,Golubeva2}).
It will be combined with Siegel's lower bound for $r_{3}\left(n\right)$,
the number of representations of $n$ as a sum of three squares ($n\not\equiv0,4,7\,\left(8\right)$)
\cite{Siegel}, which was also a key ingredient in proving the analogue
result in configuration space in \cite{Yesha}.

We comment that all this is very different from the two-dimensional
case, where arithmetic questions about spheres are replaced by questions
about circles. The two-dimensional analogue to Duke's estimate is
the theorem of Erdős and Hall \cite{Erdos} about the distribution
of lattice points on circles; an obvious complication constructing
the density one sequence $\Lambda_{\infty}$ of perturbed eigenvalues
is due to the fact that the theorem of Erdős and Hall is not deterministic,
i.e. lattice points on circles are equidistributed only for an unspecified
density one set of compatible $n$'s.

\subsection*{Acknowledgments:}

This work is part of the author's Ph.D. thesis written under the supervision
of Zeev Rudnick at Tel Aviv University. The author would like to thank
him for suggesting the problem and for his useful ideas, discussions
and comments. The author would also like to thank Stéphane Nonnenmacher
for discussions concerning pseudo-differential operators.

The research leading to these results has received funding from the
European Research Council under the European Union's Seventh Framework
Programme (FP7/2007-2013) / ERC grant agreement nº 320755.

\section{Point Scatterers on the Torus}

Let $\mathbb{T}^{3}=\mathbb{R}^{3}/2\pi\mathbb{Z}^{3}$ be the standard
flat torus. A rigorous definition of the operator \eqref{eq:Scatterer}
can be found in \cite{Yesha}, following \cite{Colin1,Rudnick}. For
the convenience of the reader we give here a brief summary:

Let $D_{0}=C_{0}^{\infty}\left(\mathbb{T}^{3}\setminus\left\{ x_{0}\right\} \right)$
be the domain of $C^{\infty}$ functions vanishing in a neighborhood
of $x_{0}$, and define an operator on $L^{2}\left(\mathbb{T}^{3}\right)$
by $-\Delta_{x_{0}}=-\Delta_{|D_{0}}$. For the adjoint of $-\Delta_{x_{0}}$
we have
\begin{align*}
\mbox{Dom}\left(-\Delta_{x_{0}}^{*}\right) & =H^{2}\left(\mathbb{T}^{3}\setminus\left\{ x_{0}\right\} \right)\\
 & =\left\{ f\in L^{2}\left(\mathbb{T}^{3}\right):\,\exists A\in\mathbb{C},\,-\Delta f+A\delta_{x_{0}}\in L^{2}\left(\mathbb{T}^{3}\right)\right\} ,
\end{align*}
and the self-adjoint extensions of $-\Delta_{x_{0}}$ are indexed
by a parameter $\phi\in(-\pi,\pi]$; the domain of the corresponding
operators $-\Delta_{\phi,x_{0}}$ contains the functions $f\in\mbox{Dom}\left(-\Delta_{x_{0}}^{*}\right)$
such that
\[
\exists a\in\mathbb{C},\, f\left(x\right)=a\left(\cos\frac{\phi}{2}\cdot\frac{-1}{4\pi\left|x-x_{0}\right|}+\sin\frac{\phi}{2}\right)+o\left(1\right),\hspace{1em}x\to x_{0}
\]
and their action on $f\in\mbox{Dom}\left(-\Delta_{\phi,x_{0}}\right)$
is given by
\begin{equation}
-\Delta_{\phi,x_{0}}f=-\Delta f+A\delta_{x_{0}}=-\Delta f+a\cos\frac{\phi}{2}\delta_{x_{0}}\label{eq:ScattererFormal}
\end{equation}
so we define a point scatterer to be one of these extended operators
$-\Delta_{\phi,x_{0}}$.

We remark that for $\phi=\pi$ we have $\mbox{Dom}\left(-\Delta_{\pi,x_{0}}\right)=H^{2}\left(\mathbb{T}^{3}\right)$
and $-\Delta_{\pi,x_{0}}f=-\Delta f$, so this extension recovers
the standard Laplacian $-\Delta_{\infty}$ on the domain $H^{2}\left(\mathbb{T}^{3}\right)$
(which is the unique self-adjoint extension of $-\Delta_{|C^{\infty}\left(\mathbb{T}^{3}\right)})$.

The functions $\left(2\pi\right)^{-3/2}e^{i\xi\cdot x}$ $\left(\xi\in\mathbb{Z}^{3}\right)$
form an orthonormal basis of eigenfunctions of $-\Delta_{\infty}$
for $L^{2}\left(\mathbb{T}^{3}\right)$. The corresponding eigenvalues
are the norms $\left|\xi\right|^{2}$ of vectors in $\mathbb{Z}^{3}$,
i.e. the set $\mathcal{N}_{3}$ of integers which are sums of three
squares, and each eigenvalue is of multiplicity $r_{3}\left(n\right)$
which is the number of representations of $n=a^{2}+b^{2}+c^{2}$ with
$a,b,c\in\mathbb{Z}$ integers.

For the perturbed operator \eqref{eq:ScattererFormal} with $\phi\neq\pi$
we still have the nonzero eigenvalues from the unperturbed problem
$\left(0\neq\lambda\in\sigma\left(-\Delta_{\infty}\right)\right)$,
with multiplicities decreased by one, as well as a new set $\Lambda=\Lambda_{\phi}$
of eigenvalues (referred as the perturbed eigenvalues), each appearing
with multiplicity one, with the corresponding eigenfunctions being
multiples of the Green's function
\[
G_{\lambda}\left(x,x_{0}\right)=\left(\Delta+\lambda\right)^{-1}\delta_{x_{0}}.
\]
The main tool we use for studying the Green's functions $G_{\lambda}$
is their $L^{2}$-expansion, which is given by
\[
G_{\lambda}(x,x_{0})=-\frac{1}{8\pi^{3}}\sum\limits _{\xi\in\mathbb{Z}^{3}}\frac{\exp\left(i\xi\cdot\left(x-x_{0}\right)\right)}{\left|\xi\right|^{2}-\lambda}.
\]
We denote by
\[
g_{\lambda}\left(x\right):=\frac{G_{\lambda}\left(x,x_{0}\right)}{\left\Vert G_{\lambda}\right\Vert _{2}}
\]
the $L^{2}$-normalized Green's function, or the normalized perturbed
eigenfunctions of the scatterer.

One can see that perturbed eigenvalues are the solutions to the equation
\begin{equation}
\sum_{\xi\in\mathbb{Z}^{3}}\left\{ \frac{1}{\left|\xi\right|^{2}-\lambda}-\frac{\left|\xi\right|^{2}}{\left|\xi\right|^{4}+1}\right\} =c_{0}\tan\frac{\phi}{2}\label{eq:eigenvalues}
\end{equation}
with
\[
c_{0}=\sum_{\xi\in\mathbb{Z}^{3}}\frac{1}{\left|\xi\right|^{4}+1}
\]
so the perturbed eigenvalues interlace with the elements of
\[
\mathcal{N}_{3}=\left\{ 0=n_{0}<n_{1}<n_{2}<\dots\right\} ,
\]
and denoting them by $\lambda_{k}=\lambda_{k}^{\phi}$ we can write
\[
\lambda_{0}<n_{0}<\lambda_{1}<n_{1}<\lambda_{2}<\dots<n_{k}<\lambda_{k+1}<n_{k+1}<\dots.
\]
Recall that a subset $\Lambda'=\left\{ \lambda_{j_{k}}\right\} \subseteq\Lambda$
is of density $a$ ($0\leq a\leq1)$ in $\Lambda$ if
\[
\lim_{J\to\infty}\frac{1}{J}\#\left\{ k\in\mathbb{N}:\, j_{k}\leq J\right\} =a
\]
or equivalently
\[
\lim_{X\to\infty}\frac{\#\left\{ \lambda\in\Lambda':\,\lambda\leq X\right\} }{\#\left\{ \lambda\in\Lambda:\,\lambda\leq X\right\} }=a.
\]

\section{\label{sec:PDO}Pseudo-Differential Operators on the Torus}

In order to quantize observables $a\in C^{\infty}\left(S^{*}\mathbb{T}^{3}\right)$,
one can use the special structure of $\mathbb{T}^{3}$ as a compact
group, to get a global definition of pseudo-differential operators
on $\mathbb{T}^{3}$, instead of using the theory of (Hörmander's)
pseudo-differential operators on $\mathbb{R}^{3}$ locally, which
could be rather inconvenient. This idea goes back to Agranovich \cite{Agranovich},
and it was proved by McLean \cite{McLean} that both definitions of
pseudo-differential operators on the torus are equivalent. A recent
monograph by Ruzhansky and Turunen \cite{Ruzhansky} gives a very
comprehensive treatment of this subject -- we present here the basic
definitions of pseudo-differential operators on $\mathbb{T}^{n}$
(and in particular on $\mathbb{T}^{3}$) using their notations.

Let $\sigma:\mathbb{Z}^{n}\to\mathbb{C}$, and let $e_{j}$ be the
standard basis elements of $\mathbb{R}^{n}$. Let $\Delta_{\xi_{j}}$
be the partial difference operator defined by
\[
\Delta_{\xi_{j}}\sigma\left(\xi\right)=\sigma\left(\xi+e_{j}\right)-\sigma\left(\xi\right),
\]
and for a multi-index $\alpha$ define
\[
\Delta_{\xi}^{\alpha}=\Delta_{\xi_{1}}^{\alpha_{1}}\cdots\Delta_{\xi_{n}}^{\alpha_{n}}.
\]
Following the notation in \cite{Ruzhansky}, for $m\in\mathbb{R},$
$0\leq\delta,\rho\leq1$, define $S_{\rho,\delta}^{m}\left(\mathbb{T}^{n}\times\mathbb{Z}^{n}\right)$
to be the set of functions $a\left(x,\xi\right)$ which are smooth
in $x$ for all $\xi\in\mathbb{Z}^{n}$, and satisfy
\[
\left|\Delta_{\xi}^{\alpha}\partial_{x}^{\beta}a\left(x,\xi\right)\right|\leq C_{a,\alpha,\beta,m}\left\langle \xi\right\rangle ^{m-\rho\left|\alpha\right|+\delta\left|\beta\right|}
\]
for every $x\in\mathbb{T}^{n}$, $\alpha,\beta$ multi-indices, and
$\xi\in\mathbb{Z}^{n}.$ Here $\left\langle \xi\right\rangle =\left(1+\left|\xi\right|^{2}\right)^{1/2}.$

For every symbol $a\left(x,\xi\right)\in S_{\rho,\delta}^{m}\left(\mathbb{T}^{n}\times\mathbb{R}^{n}\right)$,
define a toroidal symbol $\widetilde{a}:\mathbb{T}^{n}\times\mathbb{Z}^{n}\to\mathbb{C}$
by the restriction $\widetilde{a}=a|_{\mathbb{T}^{n}\times\mathbb{Z}^{n}}$;
it is not hard to show that $\widetilde{a}\left(x,\xi\right)\in S_{\rho,\delta}^{m}\left(\mathbb{T}^{n}\times\mathbb{Z}^{n}\right)$.
Define
\begin{align*}
\mbox{Op}\left(a\right)f\left(x\right) & =\mbox{Op}\left(\widetilde{a}\right)f\left(x\right)\\
 & =\sum_{\xi\in\mathbb{Z}^{n}}e^{ix\cdot\xi}a\left(x,\xi\right)\widehat{f}\left(\xi\right).
\end{align*}

One verifies (see \cite{Ruzhansky}) that the operator $\mbox{Op}\left(a\right):C^{\infty}\left(\mathbb{T}^{n}\right)\to C^{\infty}\left(\mathbb{T}^{n}\right)$
is well defined and continuous. Note that for a symbol $a\left(x,\xi\right)=\sum\limits _{\left|\alpha\right|\leq m}a_{\alpha}\left(x\right)\xi^{\alpha}$
we get that
\begin{align*}
\mbox{Op}\left(a\right)f\left(x\right) & =\sum_{\left|\alpha\right|\leq m}a_{\alpha}\left(x\right)\sum_{\xi\in\mathbb{Z}^{3}}e^{ix\cdot\xi}\xi^{\alpha}\widehat{f}\left(\xi\right)\\
 & =\sum_{\left|\alpha\right|\leq m}a_{\alpha}\left(x\right)\sum_{\xi\in\mathbb{Z}^{3}}e^{ix\cdot\xi}\widehat{\left(-i\partial\right)^{\alpha}f}\left(\xi\right)\\
 & =\sum_{\left|\alpha\right|\leq m}a_{\alpha}\left(x\right)\left(-i\partial\right)^{\alpha}f\left(x\right)
\end{align*}
as one would expect.

In particular, we use this quantization for zero-th order positively
homogeneous symbols $a\left(x,\xi\right)\in S_{1,0}^{0}\left(\mathbb{T}^{3}\times\mathbb{R}^{3}\right)$
(in the sense that $a\left(x,\lambda\xi\right)=a\left(x,\xi\right)$
for $\left|\xi\right|\geq 1$ and $\lambda\geq1$). Since $a\left(x,\xi\right)$
is a smooth $\xi$-homogeneous function of order zero,
we can identify it with a smooth function on the unit cotangent bundle
of $\mathbb{T}^{3}$, and we write $a\in C^{\infty}\left(S^{*}\mathbb{T}^{3}\right)$.

One proves that in this case, $\mbox{Op}\left(a\right)$ extends to
a bounded operator on $L^{2}\left(\mathbb{T}^{3}\right).$ Moreover,
we have the following result on the $L^{2}$-norm of $\mbox{Op}\left(a\right)$
(theorem 4.8.1 in \cite{Ruzhansky}): Let $k\in\mathbb{N}$ and $k>n/2$.
Let $a:\mathbb{T}^{n}\times\mathbb{Z}^{n}\to\mathbb{C}$ be such that
$\left|\partial_{x}^{\beta}a\left(x,\xi\right)\right|\leq C$ for
all $\left(x,\xi\right)\in\mathbb{T}^{n}\times\mathbb{Z}^{n}$ and
all $\left|\beta\right|\leq k$. Then the operator $\mbox{Op}\left(a\right)$
extends to a bounded linear operator on $L^{2}\left(\mathbb{T}^{n}\right),$
and there exists a constant $C_{k}$ (which depends only on $k$)
such that
\[
\left\Vert \mbox{Op}\left(a\right)\right\Vert _{L^{2}\left(\mathbb{T}^{n}\right)\to L^{2}\left(\mathbb{T}^{n}\right)}^{2}\leq C_{k}\sum_{\left|\alpha\right|\leq k}\sup_{y\in\mathbb{T}^{n}}\sup_{\xi\in\mathbb{Z}^{n}}\left|\partial_{y}^{\alpha}a\left(y,\xi\right)\right|^{2}.
\]
 Thus, for $n=3,\, k=2,\, a\in S_{1,0}^{0}\left(\mathbb{T}^{3}\times\mathbb{R}^{3}\right)$
we have
\begin{equation}
\left\Vert \mbox{Op}\left(a\right)\right\Vert _{L^{2}\left(\mathbb{T}^{3}\right)\to L^{2}\left(\mathbb{T}^{3}\right)}^{2}\leq C\sum_{\left|\alpha\right|\leq2}\sup_{y\in\mathbb{T}^{3}}\sup_{\xi\in\mathbb{Z}^{3}}\left|\partial_{y}^{\alpha}a\left(y,\xi\right)\right|^{2}.\label{eq:L2Bound}
\end{equation}

\section{Bounds For the Green's Functions and Truncation}

We collect some auxiliary results proved in \cite{Yesha}:

We have the following lower bound for the $L^{2}$-norm of the Green's
function:
\[
\left\Vert G_{\lambda}\right\Vert _{2}^{2}\gg\lambda^{1/2-\varepsilon}.
\]
For $L>0$, define the truncated Green's function by
\[
G_{\lambda,L}\left(x\right)=-\frac{1}{8\pi^{3}}\sum_{\left|\left|\xi\right|^{2}-\lambda\right|<L}\frac{e^{i\xi\cdot\left(x-x_{0}\right)}}{\left|\xi\right|^{2}-\lambda}
\]
and the $L^{2}$-normalized truncated Green's function by $g_{\lambda,L}=\frac{G_{\lambda,L}}{\left\Vert G_{\lambda,L}\right\Vert _{2}}.$

For $L=\lambda^{\delta},$ $\delta>0$, we have
\[
\left\Vert g_{\lambda}-g_{\lambda,L}\right\Vert _{2}\to0
\]
 as $\lambda\to\infty$ and
\[
\left\Vert G_{\lambda,L}\right\Vert _{2}=\left\Vert G_{\lambda}\right\Vert _{2}\left(1+o\left(1\right)\right).
\]
It is also proved that for all $f\in C^{\infty}\left(\mathbb{T}^{3}\right)$
\[
\left|\left\langle fg_{\lambda},g_{\lambda}\right\rangle -\left\langle fg_{\lambda,L},g_{\lambda,L}\right\rangle \right|\leq2\left\Vert f\right\Vert _{\infty}\left\Vert g_{\lambda}-g_{\lambda,L}\right\Vert _{2},
\]
and a similar proof shows that for all $a\in C^{\infty}\left(S^{*}\mathbb{T}^{3}\right)$
\begin{gather}
\left|\left\langle \mbox{Op}\left(a\right)g_{\lambda},g_{\lambda}\right\rangle -\left\langle \mbox{Op}\left(a\right)g_{\lambda,L},g_{\lambda,L}\right\rangle \right|\leq\label{eq:TruncationBound}\\
2\left\Vert \mbox{Op}\left(a\right)\right\Vert _{L^{2}\left(\mathbb{T}^{3}\right)\to L^{2}\left(\mathbb{T}^{3}\right)}\left\Vert g_{\lambda}-g_{\lambda,L}\right\Vert _{2}.\nonumber
\end{gather}

\section{Basis Elements}

The following lemma shows that we can approximate a smooth function
on $\mathbb{T}^{3}\times S^{2}$ by a finite linear combination of
functions of the form $e_{\zeta,l,m}\left(x,\xi\right)=Y_{l,m}\left(\xi\right)e^{ix\cdot\zeta}$,
where $Y_{l,m}\left(\xi\right)$ is the spherical harmonic of degree
$l$ and order $m$ (normalized so that $\int_{S^{2}}\left|Y_{l,m}\right|^{2}\mbox{d}\sigma=1$).
\begin{lem}
\label{lem:PolApprox}For all $a\in C^{\infty}\left(\mathbb{T}^{3}\times S^{2}\right)$,
$\varepsilon>0$, there exist $N_{1},\, N_{2}$ and
\[
P\left(x,\xi\right)=\sum\limits _{\left|\zeta\right|\leq N_{1}}\sum\limits _{l\leq N_{2}}\sum_{\left|m\right|\leq l}c_{\zeta,l,m}e_{\zeta,l,m}\left(x,\xi\right)\in C^{\infty}\left(\mathbb{T}^{3}\times S^{2}\right),
\]
such that for all $x\in\mathbb{T}^{3},\,\xi\in S^{2}$ and for all
multi-index $\alpha$ such that $\left|\alpha\right|\leq2$, we have
$\left|\partial_{x}^{\alpha}\left(a-P\right)\left(x,\xi\right)\right|<\varepsilon$.\end{lem}
\begin{proof}
Let $\varepsilon>0$. Expanding $a$ to its Fourier series in the
variable $x$, for every $\xi\in S^{2}$ we have
\begin{equation}
a\left(x,\xi\right)=\sum\limits _{\zeta\in\mathbb{Z}^{3}}a_{\zeta}\left(\xi\right)e^{ix\cdot\zeta}\label{eq:FourierX}
\end{equation}
where $a_{\zeta}\left(\xi\right)=\frac{1}{8\pi^{3}}\int_{\mathbb{T}^{3}}a\left(x,\xi\right)e^{-ix\cdot\zeta}\mbox{d}x$,
with convergence in the sense of $L^{2}\left(\mathbb{T}^{3}\right)$.
Moreover, since $a\in C^{\infty}\left(\mathbb{T}^{3}\times S^{2}\right)$,
integration by parts yields that $a_{\zeta}\left(\xi\right)\ll\left|\zeta\right|^{-k}$
for all $k$ (where the implied constant is independent of $\xi$),
so the series in \eqref{eq:FourierX} is uniformly convergent in $\mathbb{T}^{3}\times S^{2}$.
Since uniform convergence implies $L^{2}$ convergence to the same
(equivalence class of) function, and since for every $\xi$ both sides
of \eqref{eq:FourierX} are continuous functions on $\mathbb{T}^{3}$,
we conclude that the series uniformly converges in $\mathbb{T}^{3}\times S^{2}$
to $a\left(x,\xi\right)$. We also get that for every multi-index
$\alpha$, we have
\begin{align}
\partial_{x}^{\alpha}a\left(x,\xi\right) & =\sum\limits _{\zeta\in\mathbb{Z}^{3}}a_{\zeta}\left(\xi\right)\partial_{x}^{\alpha}e^{ix\cdot\zeta}\nonumber \\
 & =\sum\limits _{\zeta\in\mathbb{Z}^{3}}a_{\zeta}\left(\xi\right)\left(i\zeta\right)^{\alpha}e^{ix\cdot\zeta}\label{eq:FourierXPartial}
\end{align}
and the series in \eqref{eq:FourierXPartial} is uniformly convergent
in $\mathbb{T}^{3}\times S^{2}$ to $\partial_{x}^{\alpha}a\left(x,\xi\right)$,
so we can find $N_{1}$ such that for all $x,\xi$ and for all multi-index
$\alpha$ such that $\left|\alpha\right|\leq2$ we have
\[
\left|\partial_{x}^{\alpha}\left(a\left(x,\xi\right)-\sum\limits _{\left|\zeta\right|\leq N_{1}}a_{\zeta}\left(\xi\right)e^{ix\cdot\zeta}\right)\right|<\frac{\varepsilon}{2}.
\]
For every $\zeta$, we have the spherical harmonics expansion:
\begin{equation}
a_{\zeta}\left(\xi\right)=\sum_{l=0}^{\infty}\sum_{\left|m\right|\leq l}c_{\zeta,l,m}Y_{l,m}\left(\xi\right)\label{eq:SphericalExpansion}
\end{equation}
where $c_{\zeta,l,m}=\int_{S^{2}}a_{\zeta}\left(\xi\right)\overline{Y_{l,m}\left(\xi\right)}\mbox{d}\sigma$,
with convergence in the sense of $L^{2}\left(S^{2}\right)$. Since
$a\left(x,\xi\right)\in C^{\infty}\left(\mathbb{T}^{3}\times S^{2}\right)$,
we easily see that for all $\zeta$ we have $a_{\zeta}\left(\xi\right)\in C^{\infty}\left(S^{2}\right),$
and hence $\sum\limits _{\left|m\right|\leq l}c_{\zeta,l,m}Y_{l,m}\left(\xi\right)\ll l^{-k}$
for all $k$ (see \cite{Atkinson}, for example), so the sum in \eqref{eq:SphericalExpansion}
is uniformly convergent in $S^{2}$ for all $\zeta$, and again, it
must converge to $a_{\zeta}\left(\xi\right)$. We conclude that there
exists $N_{2}$ such that for all $\left|\zeta\right|\leq N_{1}$
and for all $\xi$ we have
\[
\left|a_{\zeta}\left(\xi\right)-\sum_{l\leq N_{2}}\sum_{\left|m\right|\leq l}c_{\zeta,l,m}Y_{l,m}\left(\xi\right)\right|<\frac{\varepsilon}{2N_{1}^{5}}
\]
and if we denote
\[
P\left(x,\xi\right)=\sum\limits _{\left|\zeta\right|\leq N_{1}}\sum\limits _{l\leq N_{2}}\sum\limits _{\left|m\right|\leq l}c_{\zeta,l,m}e_{\zeta,l,m}\left(x,\xi\right)
\]
we get that for all $x,\xi$ and for all multi-index $\alpha$ such
that $\left|\alpha\right|\leq2$ we have
\begin{gather*}
\left|\partial_{x}^{\alpha}\left(a-P\right)\left(x,\xi\right)\right|\leq\left|\partial_{x}^{\alpha}\left(P\left(x,\xi\right)-\sum\limits _{\left|\zeta\right|\leq N_{1}}a_{\zeta}\left(\xi\right)e^{ix\cdot\zeta}\right)\right|+\frac{\varepsilon}{2}\\
=\left|\partial_{x}^{\alpha}\left(\sum\limits _{\left|\zeta\right|\leq N_{1}}\left(a_{\zeta}\left(\xi\right)-\sum_{l\leq N_{2}}\sum_{\left|m\right|\leq l}c_{\zeta,l,m}Y_{l,m}\left(\xi\right)\right)e^{ix\cdot\zeta}\right)\right|+\frac{\varepsilon}{2}\\
\leq N_{1}^{2}\sum\limits _{\left|\zeta\right|\leq N_{1}}\left|a_{\zeta}\left(\xi\right)-\sum_{l\leq N_{2}}\sum_{\left|m\right|\leq l}c_{\zeta,l,m}Y_{l,m}\left(\xi\right)\right|+\frac{\varepsilon}{2}<\varepsilon.
\end{gather*}

\end{proof}
We can think of $e_{\zeta,l,m}\left(x,\xi\right)=Y_{l,m}\left(\xi\right)e^{ix\cdot\zeta}$
as a zero-th order positively homogeneous symbol in $S_{1,0}^{0}\left(\mathbb{T}^{3}\times\mathbb{R}^{3}\right)$,
by extending $Y_{l,m}\left(\xi\right),$ $\left(l,m\right)\neq\left(0,0\right)$
homogeneously to the domain $\left|\xi\right|\geq 1$, and arbitrarily
to the domain $\left|\xi\right|< 1$; the function $Y_{0,0}\left(\xi\right)=\frac{1}{2}\sqrt{\frac{1}{\pi}}$
extends to $\xi\in\mathbb{R}^{3}$ in an obvious way . In order to
prove our main theorem, we will now see that it suffices to show a
simpler version of it on the functions $e_{\zeta,l,m}\left(x,\xi\right)$.
\begin{prop}
Suppose that there is a subset $\Lambda_{\infty}\subseteq\Lambda$
of density one so that for all $\zeta,l,m$ with at least one of them
nonzero we have
\[
\left\langle \mbox{Op}\left(e_{\zeta,l,m}\right)g_{\lambda,L},g_{\lambda,L}\right\rangle \to0
\]
as $\lambda\to\infty$ along $\Lambda_{\infty}$, then theorem \ref{thm:MainTheroem}
follows. Here $0<\delta<1$, $L=\lambda^{\delta}.$\end{prop}
\begin{proof}
Let $\varepsilon>0$, and $a\left(x,\xi\right)\in C^{\infty}\left(S^{*}\mathbb{T}^{3}\right)$.
From lemma \ref{lem:PolApprox}, there is
\[
P\left(x,\xi\right)=\sum\limits _{\left|\zeta\right|\leq N_{1}}\sum\limits _{l\leq N_{2}}\sum\limits _{\left|m\right|\leq l}c_{\zeta,l,m}e_{\zeta,l,m}\left(x,\xi\right)
\]
such that for all $x\in\mathbb{T}^{3}$,~$\xi\in\mathbb{R}^{3},\,\left|\xi\right|\geq 1$
and for all multi-index $\alpha$ such that $\left|\alpha\right|\leq2$
we have $\left|\partial_{x}^{\alpha}\left(a-P\right)\left(x,\xi\right)\right|<\varepsilon$.
Without loss of generality we can assume that we have
\begin{equation}
\partial_{x}^{\alpha}P\left(x,0\right)=\partial_{x}^{\alpha}a\left(x,0\right)\label{eq:WLOGPol}
\end{equation}
 for all $x\in\mathbb{T}^{3}$ and for all multi-index $\alpha$,
because for $\lambda$ large enough
\[
\mbox{Op}\left(P\right)g_{\lambda,L}\left(x\right)=-\frac{1}{8\pi^{3}\left\Vert G_{\lambda,L}\right\Vert _{2}}\sum_{\left|\left|\xi\right|^{2}-\lambda\right|<L}e^{i\left(x-x_{0}\right)\cdot\xi}P\left(x,\xi\right)\frac{1}{\left|\xi\right|^{2}-\lambda}
\]
does not depend on the value of $P\left(x,0\right)$
(and it is easy to change $P\left(x,\xi\right)$ in the domain $\left|\xi\right|<1$
to get a new symbol $\tilde{P}\in C^{\infty}\left(S^{*}\mathbb{T}^{3}\right)$
satisfying \eqref{eq:WLOGPol}), so under this assumption the inequality
$\left|\partial_{x}^{\alpha}\left(a-P\right)\left(x,\xi\right)\right|<\varepsilon$
holds for every $\xi\in\mathbb{Z}^{n}$. Since
\[
\mbox{Op}\left(e_{0,0,0}\right)=\frac{1}{2}\sqrt{\frac{1}{\pi}}\mbox{id},
\]
and since for $\left(\zeta,l,m\right)\neq\left(0,0,0\right)$
\[
\left\langle \mbox{Op}\left(e_{\zeta,l,m}\right)g_{\lambda,L},g_{\lambda,L}\right\rangle \to0
\]
as $\lambda\to\infty$ along $\Lambda_{\infty}$, we have
\begin{align*}
\left\langle \mbox{Op}\left(P\right)g_{\lambda,L},g_{\lambda,L}\right\rangle  & =\sum\limits _{\left|\zeta\right|\leq N_{1}}\sum\limits _{l\leq N_{2}}\sum\limits _{\left|m\right|\leq l}c_{\zeta,l,m}\left\langle \mbox{Op}\left(e_{\zeta,l,m}\right)g_{\lambda,L},g_{\lambda,L}\right\rangle \\
 & \to\frac{1}{2}\sqrt{\frac{1}{\pi}}c_{0,0,0}
\end{align*}
as $\lambda\to\infty$ along $\Lambda_{\infty}$, and
\begin{align*}
\int_{S^{*}\mathbb{T}^{3}}P\mbox{d}\mu & =\frac{1}{\mbox{area}\left(\mathbb{T}^{3}\right)\mbox{area}\left(S^{2}\right)}\\
 & \sum\limits _{\left|\zeta\right|\leq N_{1}}\sum\limits _{l\leq N_{2}}\sum\limits _{\left|m\right|\leq l}c_{\zeta,l,m}\left(\int_{\mathbb{T}^{3}}e^{ix\cdot\zeta}\mbox{d}m\right)\left(\int_{S^{2}}Y_{l,m}\left(\xi\right)\mbox{d}\sigma\right)\\
 & =\frac{1}{\mbox{area}\left(\mathbb{T}^{3}\right)\mbox{area}\left(S^{2}\right)}c_{0,0,0}\left(\int_{\mathbb{T}^{3}}1\mbox{d}m\right)\left(\int_{S^{2}}Y_{0,0}\left(\xi\right)\mbox{d}\sigma\right)\\
 & =\frac{1}{2}\sqrt{\frac{1}{\pi}}c_{0,0,0}
\end{align*}
so
\[
\left\langle \mbox{Op}\left(P\right)g_{\lambda,L},g_{\lambda,L}\right\rangle \to\int_{S^{*}\mathbb{T}^{3}}P\mbox{d}\mu
\]
as $\lambda\to\infty$ along $\Lambda_{\infty}$. Thus for $\lambda\in\Lambda_{\infty}$
large enough we have
\begin{align*}
\left|\left\langle \mbox{Op}\left(P\right)g_{\lambda},g_{\lambda}\right\rangle -\int_{S^{*}\mathbb{T}^{3}}P\mbox{d}\mu\right| & <\varepsilon+\left|\left\langle \mbox{Op}\left(P\right)g_{\lambda},g_{\lambda}\right\rangle -\left\langle \mbox{Op}\left(P\right)g_{\lambda,L},g_{\lambda,L}\right\rangle \right|\\
 & \leq\varepsilon+2\left\Vert \mbox{Op}\left(P\right)\right\Vert _{L^{2}\left(\mathbb{T}^{3}\right)\to L^{2}\left(\mathbb{T}^{3}\right)}\left\Vert g_{\lambda}-g_{\lambda,L}\right\Vert _{2}\\
 & <C\varepsilon.
\end{align*}

Using the bound in \eqref{eq:L2Bound}, we get that
\begin{align*}
\left|\left\langle \mbox{Op}\left(a-P\right)g_{\lambda},g_{\lambda}\right\rangle \right|^{2} & \leq\left\Vert \mbox{Op}\left(a-P\right)g_{\lambda}\right\Vert _{2}^{2}\\
 & \leq\left\Vert \mbox{Op}\left(a-P\right)\right\Vert _{L^{2}\left(\mathbb{T}^{3}\right)\to L^{2}\left(\mathbb{T}^{3}\right)}^{2}\\
 & \leq C\sum_{\left|\alpha\right|\leq2}\sup_{y\in\mathbb{T}^{3}}\sup_{\xi\in\mathbb{Z}^{3}}\left|\partial_{y}^{\alpha}\left(a-P\right)\left(y,\xi\right)\right|^{2}\\
 & <C\varepsilon^{2}
\end{align*}
for $\lambda\in\Lambda_{\infty}$ (we call all our constants $C$).

We conclude that for $\lambda\in\Lambda_{\infty}$ large enough we
have
\begin{align*}
\left|\left\langle \mbox{Op}\left(a\right)g_{\lambda},g_{\lambda}\right\rangle -\int_{S^{*}\mathbb{T}^{3}}a\mbox{d}\mu\right| & \leq\left|\left\langle \mbox{Op}\left(a-P\right)g_{\lambda},g_{\lambda}\right\rangle \right|\\
 & +\left|\left\langle \mbox{Op}\left(P\right)g_{\lambda},g_{\lambda}\right\rangle -\int_{S^{*}\mathbb{T}^{3}}P\mbox{d}\mu\right|\\
 & +\left|\int_{S^{*}\mathbb{T}^{3}}P\mbox{d}\mu-\int_{S^{*}\mathbb{T}^{3}}a\mbox{d}\mu\right|\\
 & <C\varepsilon
\end{align*}
and the proposition follows.
\end{proof}

\section{A Density One Set}

By the theorem of Legendre and Gauss (see \cite{Grosswald}), the
Diophantine equation
\[
x_{1}^{2}+x_{2}^{2}+x_{3}^{2}=n
\]
has solutions in integers $x_{i}$ $\left(i=1,2,3\right)$ if and
only if $n$ is not of the form $4^{a}\left(8k+7\right)$ with $a\in\mathbb{Z}$,
$a\geq0$ and $k\in\mathbb{Z}$, and for all $n$, $r_{3}\left(4^{a}n\right)=r_{3}\left(n\right).$

Equivalently, if we write $n=4^{a}n_{1},$ with $4\nmid n_{1}$, then
$n$ is a sum of three squares if and only if $n_{1}\not\equiv7\,\left(8\right)$,
that is to say
\[
\mathcal{N}_{3}=\left\{ n\in\mathbb{N}:\, n=4^{a}n_{1},\,4\nmid n_{1}\Rightarrow n_{1}\not\equiv7\,\left(8\right)\right\} ,
\]
and $r_{3}\left(n\right)=r_{3}\left(n_{1}\right)$.

For $n\in\mathcal{N}_{3}$, write $n=4^{a}n_{1}$ with $4\nmid n_{1}$,
and define
\[
\mathcal{N}_{good}=\left\{ n\in\mathcal{N}_{3}:\, n_{1}>n^{1/2}\right\}
\]
the set of ``good'' elements in $\mathcal{N}_{3}$, and $\mathcal{N}_{bad}=\mathcal{N}_{3}\setminus\mathcal{N}_{good}$
the set of ``bad elements''. We show that there are very few ``bad''
elements in $\mathcal{N}_{3}$:
\begin{lem}
\label{lem:NBadEst}$\#\left\{ n\in\mathcal{N}_{bad}:\, n\leq X\right\} \leq X^{1/2}\log X$.\end{lem}
\begin{proof}
For $n\in\mathcal{N}_{bad}$, $n\leq X$, we have $n=4^{a}n_{1}$
with $n_{1}\leq n^{1/2}\leq X^{1/2}$ and $a=\log_{4}\left(n/n_{1}\right)\leq\log n/\log4\leq\log X$,
so there are at most $X^{1/2}\log X$ possibilities for such $n$.
\end{proof}
For $\lambda\in\Lambda$, denote $n_{\lambda}$ to be the closest
element in $\mathcal{N}_{3}$ to $\lambda$ (and if there are two
elements with the same distance take the smallest of them). Note that
$\left|n_{\lambda}-\lambda\right|\leq1.5$, and for $n\neq n_{\lambda}$
we have $\left|n-\lambda\right|\geq0.5$.

Define $\Lambda_{\infty}=\left\{ \lambda\in\Lambda:\, n_{\lambda}\in\mathcal{N}_{good}\right\} .$
We show that:
\begin{lem}
$\Lambda_{\infty}$ is a density one set in $\Lambda$.\end{lem}
\begin{proof}
For every $\lambda\in\Lambda\setminus\Lambda_{\infty}$, we have $n_{k}<\lambda<n_{k+1}$,
where either $n_{\lambda}=n_{k}$ or $n_{\lambda}=n_{k+1}$, and $n_{\lambda}\in\mathcal{N}_{bad}$.
Thus, for every $n\in\mathcal{N}_{bad}$ such that $n\leq X+1.5$
there are at most two $\lambda\in\Lambda\setminus\Lambda_{\infty}$,
$\lambda\leq X$ such that $n_{\lambda}=n$, so by lemma \ref{lem:NBadEst}
we have
\begin{align*}
\#\left\{ \lambda\in\Lambda\setminus\Lambda_{\infty}:\,\lambda\leq X\right\}  & \leq2\#\left\{ n\in\mathcal{N}_{bad}:\, n\leq X+1.5\right\} \ll X^{1/2}\log X,
\end{align*}
but
\[
\#\left\{ \lambda\in\Lambda:\,\lambda\leq X\right\} \geq\#\left\{ n\leq X:\, n\not\equiv0,4,7\,\left(8\right)\right\} \asymp X,
\]
so
\[
\frac{\#\left\{ \lambda\in\Lambda\setminus\Lambda_{\infty}:\,\lambda\leq X\right\} }{\#\left\{ \lambda\in\Lambda:\,\lambda\leq X\right\} }\ll X^{-1/2}\log X
\]
which tends to zero as $X\to\infty$, so $\Lambda\setminus\Lambda_{\infty}$
is a density zero set in $\Lambda$, and therefore $\Lambda_{\infty}$
is a density one set in $\Lambda$.
\end{proof}

\section{Proving Theorem \ref{thm:MainTheroem}}

We are only left to prove the following proposition:
\begin{prop}
Let $\zeta,l,m$ with at least one of them nonzero, and let $0<\delta<1/28$,
$L=\lambda^{\delta}.$ Then
\[
\left\langle \mbox{Op}\left(e_{\zeta,l,m}\right)g_{\lambda,L},g_{\lambda,L}\right\rangle \to0
\]
as $\lambda\to\infty$ along $\Lambda_{\infty}$. \end{prop}
\begin{proof}
We have
\begin{gather*}
\left|\left\langle \mbox{Op}\left(e_{\zeta,l,m}\right)G_{\lambda,L},G_{\lambda,L}\right\rangle \right|\\
\asymp\left|\left\langle \sum_{\left|\left|\xi\right|^{2}-\lambda\right|<L}\frac{e^{i\left(x-x_{0}\right)\cdot\xi}}{\left|\xi\right|^{2}-\lambda}e^{ix\cdot\zeta}Y_{l,m}\left(\frac{\xi}{\left|\xi\right|}\right),\sum_{\left|\left|\eta\right|^{2}-\lambda\right|<L}\frac{e^{i\left(x-x_{0}\right)\cdot\eta}}{\left|\eta\right|^{2}-\lambda}\right\rangle \right|\\
\asymp\left|\sum_{\begin{subarray}{c}
\left|\left|\xi\right|^{2}-\lambda\right|<L\\
\left|\left|\xi+\zeta\right|^{2}-\lambda\right|<L
\end{subarray}}\frac{1}{\left(\left|\xi\right|^{2}-\lambda\right)\left(\left|\xi+\zeta\right|^{2}-\lambda\right)}Y_{l,m}\left(\frac{\xi}{\left|\xi\right|}\right)\right|.
\end{gather*}
First, assume that $\zeta\neq0:$

The functions $Y_{l,m}$ are bounded on $S^{2}$, so
\begin{align*}
\left|\left\langle \mbox{Op}\left(e_{\zeta,l,m}\right)G_{\lambda,L},G_{\lambda,L}\right\rangle \right| & \ll\sum_{\begin{subarray}{c}
\left|\left|\xi\right|^{2}-\lambda\right|<L\\
\left|\left|\xi+\zeta\right|^{2}-\lambda\right|<L
\end{subarray}}\frac{1}{\left|\left|\xi\right|^{2}-\lambda\right|\left|\left|\xi+\zeta\right|^{2}-\lambda\right|}
\end{align*}
and therefore
\[
\left|\left\langle \mbox{Op}\left(e_{\zeta,l,m}\right)g_{\lambda,L},g_{\lambda,L}\right\rangle \right|\ll\frac{\sum\limits _{\begin{subarray}{c}
\left|\left|\xi\right|^{2}-\lambda\right|<L\\
\left|\left|\xi+\zeta\right|^{2}-\lambda\right|<L
\end{subarray}}\frac{1}{\left|\left|\xi\right|^{2}-\lambda\right|\left|\left|\xi+\zeta\right|^{2}-\lambda\right|}}{\left\Vert G_{\lambda,L}\right\Vert _{2}^{2}}\to0
\]
as $\lambda\to\infty$ by the proof of Proposition 3.9 in \cite{Yesha}.

Assume now that $\zeta=0$.

We have
\begin{align*}
\left|\left\langle \mbox{Op}\left(e_{0,l,m}\right)G_{\lambda,L},G_{\lambda,L}\right\rangle \right| & \asymp\left|\sum_{\left|\left|\xi\right|^{2}-\lambda\right|<L}\frac{Y_{l,m}\left(\frac{\xi}{\left|\xi\right|}\right)}{\left(\left|\xi\right|^{2}-\lambda\right)^{2}}\right|\\
 & =\left|\sum_{\left|n-\lambda\right|<L}\frac{W_{l,m}\left(n\right)}{\left(n-\lambda\right)^{2}}\right|
\end{align*}
where
\[
W_{l,m}\left(n\right)=\sum\limits _{\left|\xi\right|^{2}=n}Y_{l,m}\left(\frac{\xi}{\left|\xi\right|}\right).
\]
If we write $n=4^{a}n_{1}$ with $4\nmid n_{1}$, then by Duke's estimate
(see \cite{Duke,Duke2})
\[
\left|W_{l,m}\left(n\right)\right|=\left|W_{l,m}\left(n_{1}\right)\right|\ll n_{1}^{13/28+\varepsilon}\leq n^{13/28+\varepsilon}
\]
and by Siegel's theorem \cite{Siegel} $r_{3}\left(n\right)=r_{3}\left(n_{1}\right)\gg n_{1}^{1/2-\varepsilon},$
so
\[
\frac{\left|W_{l,m}\left(n\right)\right|}{r_{3}\left(n\right)}\ll n_{1}^{-1/28+\varepsilon}.
\]
For $\lambda\in\Lambda_{\infty}$, recall that $n_{\lambda}$ is the
closest element in $\mathcal{N}_{3}$ to $\lambda$ (and if there
are two elements with the same distance take the smallest of them).
From the definition of $\Lambda_{\infty}$ we know that $n_{\lambda}\in\mathcal{N}_{good}$.

Write
\begin{align*}
\sum_{\left|n-\lambda\right|<L}\frac{W_{l,m}\left(n\right)}{\left(n-\lambda\right)^{2}} & =\sum_{\begin{subarray}{c}
\left|n-\lambda\right|<L\\
n\neq n_{\lambda}
\end{subarray}}\frac{W_{l,m}\left(n\right)}{\left(n-\lambda\right)^{2}}+\frac{W_{l,m}\left(n_{\lambda}\right)}{\left(n_{\lambda}-\lambda\right)^{2}}.
\end{align*}
Since for $n\neq n_{\lambda}$ we have $\left|n-\lambda\right|\geq0.5$:
\[
\left|\sum_{\begin{subarray}{c}
\left|n-\lambda\right|<L\\
n\neq n_{\lambda}
\end{subarray}}\frac{W_{l,m}\left(n\right)}{\left(n-\lambda\right)^{2}}\right|\ll\sum_{\left|n-\lambda\right|<L}n^{13/28+\varepsilon}\ll\lambda^{13/28+\delta+\varepsilon}
\]
so
\[
\frac{\left|\sum\limits _{\begin{subarray}{c}
\left|n-\lambda\right|<L\\
n\neq n_{\lambda}
\end{subarray}}\frac{W_{l,m}\left(n\right)}{\left(n-\lambda\right)^{2}}\right|}{\left\Vert G_{\lambda,L}\right\Vert _{2}^{2}}\ll\frac{\lambda^{13/28+\delta+\varepsilon}}{\left\Vert G_{\lambda}\right\Vert _{2}^{2}}\ll\lambda^{-1/28+\delta+2\varepsilon}
\]
which tends to zero for $\varepsilon>0$ small enough, since $\delta<\frac{1}{28}$.
For the other term, writing $n_{\lambda}=4^{a}n_{1}$ with $4\nmid n_{1}$,
we have
\begin{align*}
\frac{\left|\frac{W_{l,m}\left(n_{\lambda}\right)}{\left(n_{\lambda}-\lambda\right)^{2}}\right|}{\left\Vert G_{\lambda,L}\right\Vert _{2}^{2}} & \asymp\frac{\frac{\left|W_{l,m}\left(n_{\lambda}\right)\right|}{\left(n_{\lambda}-\lambda\right)^{2}}}{\left\Vert G_{\lambda}\right\Vert _{2}^{2}}\\
 & \asymp\frac{\frac{\left|W_{l,m}\left(n_{\lambda}\right)\right|}{\left(n_{\lambda}-\lambda\right)^{2}}}{\sum\limits _{n=0}^{\infty}\frac{r_{3}\left(n\right)}{\left(n-\lambda\right)^{2}}}\\
 & \ll\frac{\left|W_{l,m}\left(n_{\lambda}\right)\right|}{r_{3}\left(n_{\lambda}\right)}\ll n_{1}^{-1/28+\varepsilon}
\end{align*}
and since $n_{\lambda}\in\mathcal{N}_{good}$, we have
\[
n_{1}^{-1/28+\varepsilon}\leq n_{\lambda}^{-1/56+\frac{\varepsilon}{2}}\ll\lambda^{-1/56+\frac{\varepsilon}{2}}
\]
so this term also tends to zero as $\lambda\to\infty$ along $\Lambda_{\infty}$. \end{proof}

\end{document}